\documentclass[11pt]{amsart}
\usepackage{amssymb}
\theoremstyle{plain}
\newtheorem{theorem}{Theorem}

 \newtheorem{lemma}{Lemma}
\newtheorem{proposition}{Proposition}

\theoremstyle{example}
\newtheorem{example}{Example}

\theoremstyle{definition}

\theoremstyle{remark}

\numberwithin{equation}{section}

\setbox0=\hbox{$+$}
\newdimen\plusheight
\plusheight=\ht0
\def\+{\;\lower\plusheight\hbox{$+$}\;}

\setbox0=\hbox{$-$}
\newdimen\minusheight
\minusheight=\ht0
\def\-{\;\lower\minusheight\hbox{$-$}\;}

\setbox0=\hbox{$\cdots$}
\newdimen\cdotsheight
\cdotsheight=\plusheight
\def\cds{\lower\cdotsheight\hbox{$\cdots$}}

\begin{document}
\title[ Divergence in the General Sense for Continued Fractions ]
       {A Theorem on Divergence in the General Sense for Continued Fractions }
\author{Douglas Bowman}
\address{Department Of Mathematical Sciences,
Northern Illinois University,
 De Kalb, IL 60115}
\email{bowman@math.niu.edu }
\author{James Mc Laughlin}
\address{Mathematics Department,
       Trinity College,
       300 Summit Street,
       Hartford, CT 06106-3100}
\email{james.mclaughlin@trincoll.edu}
\thanks{The second author's research supported in part by a
Trjitzinsky Fellowship.}
\date{April, 18, 2002}
\begin{abstract}
If the odd and even parts of a continued fraction converge to
different values, the continued fraction may or may not converge
in the general sense. We prove a theorem which settles the
question of general convergence for a wide class of such continued
fractions.

We apply this theorem to two  general classes of $q$ continued
fraction to show, that if $G(q)$ is one of these continued
fractions and $|q|>1$, then either $G(q)$ converges or does not
converge in the general sense.

We also show that if the odd and even parts of the continued fraction
$K_{n=1}^{\infty}a_{n}/1$
converge to different values, then $\lim_{n \to \infty}|a_{n}| = \infty$.
\end{abstract}

\maketitle

\section{Introduction}

In \cite{J86}, Jacobsen revolutionised the subject of the convergence
of continued fractions by introducing the concept of \emph{general convergence}.
General convergence is defined in \cite{LW92} as follows.

Let the $n$-th approximant of the continued fraction
{\allowdisplaybreaks
\begin{equation}\label{Meq}
M=b_{0}+\cfrac{a_{1}}{b_{1} + \cfrac{a_{2}}{b_{2} + \cfrac{a_{3}}{b_{3}+
\cds }}}
\end{equation}
} be denoted by $A_{n}/B_{n}$ ($A_{n}$ is the $n$-th
\emph{numerator} convergent and $B_{n}$ is the $n$-th
\emph{denominator} convergent) and let
\[S_{n}(w)= \frac{A_{n}+wA_{n-1}}{B_{n}+wB_{n-1}}.
\]
Define the chordal metric $d$ on $\hat{\mathbb{C}}$ by
{\allowdisplaybreaks
\begin{equation*}
d(w,z)=\frac{|z-w|}
{\sqrt{1+|w|^{2}}\sqrt{1+|z|^{2}}}
\end{equation*}
} when $w$ and $z$ are both finite, and
\[
d(w, \infty) = \frac{1}{\sqrt{1+|w|^{2}}}.
\]

\textbf{Definition:} The continued fraction
$M$ is said to \emph{converge generally} to
$f \in \hat{\mathbb{C}}$ if there exist sequences
$\{v_{n}\}$, $\{w_{n}\} \subset \hat{\mathbb{C}}$ such that
$\liminf d(v_{n},w_{n})>0$ and
\[
\lim_{n \to \infty}S_{n}(v_{n})=\lim_{n \to \infty}S_{n}(w_{n}) = f.
\]
Remark: Jacobson shows in \cite{J86} that, if a continued fraction converges in
the general sense, then the limit is unique.

The idea of general convergence is of great significance because
classical convergence implies general convergence (take $v_{n}=0$
and $w_{n}= \infty$, for all $n$), but the converse does not
necessarily hold. General convergence is a natural extension of
the concept of classical convergence for continued fractions.

The \emph{even} part of the continued fraction $M$ at \eqref{Meq}
 is the continued fraction whose
$n$-th numerator (denominator) convergent equals $A_{2n}$
($B_{2n}$), for $n \geq 0$. The \emph{odd} part of $M$ is the
continued fraction whose zero-th numerator convergent is
$A_{1}/B_{1}$, whose zero-th denominator convergent is $1$, and
whose $n$-th numerator (respectively denominator) convergent
equals $A_{2n+1}$ (respectively $B_{2n+1}$),  for $n \geq 1$.

In this present paper we investigate the general convergence of
continued fractions whose odd and even parts each converge, but to
different values. Such continued fractions may or may not converge
in the general sense as the following examples show.

\begin{example} Let
\begin{equation}\label{rrcf}
K(q) = 1+ \frac{q}{1}
\+
\frac{q^{2}}{1}
\+
\frac{q^{3}}{1}
\+
\cds
\+
\frac{q^{n}}{1}
\+ \cds.
\end{equation}
If $|q|>1$ then the odd and even parts of $K(q)$ converge
but $K(q)$ does not converge generally.
\end{example}
The continued fraction $K(q)$ is the famous Rogers-Ramanujan
continued fraction. It was stated without proof by Ramanujan that,
if $|q|>1$, then the odd  part of $K(q)$ converges to $1/K(-1/q)$
and the even part converges to $q\,K(1/q^{4})$
 (See Entry 59 of \cite{ABJL92} for a proof of
Ramanujan's claim). However,  $K(q)$ is easily seen
to be equivalent to the following continued fraction:
\begin{equation*}
\hat{K}(q) :=  1+ \frac{1}{1/q} \+ \frac{1}{1/q} \+
\frac{1}{1/q^{2}} \+ \frac{1}{1/q^{2}} \+ \cds \+ \frac{1}{1/q^{n}}
\+ \frac{1}{1/q^{n}} \+ \cds.
\end{equation*}
It is an easy consequence of the
Stern-Stolz Theorem below, as extended by Lorentzen and Waadeland,
that this continued fraction does not converge in the general
sense for any $q$ outside the unit circle.
{\allowdisplaybreaks
\begin{example} Let
\[
G:=\frac{2}{1} \+
 \frac{-1}{2}
\+ K_{n=3}^{\infty} \frac{a_{n}}{b_{n}},
\]
where
{\allowdisplaybreaks
\begin{align*}
a_{2n+1}&= 1 + \frac{1}{2\,n^2} + \frac{1}{n},&
b_{2n+1}&=\frac{-1}{2\,n^3},&\\
a_{2n+2}&=\frac{2\,
      {\left( 1 + n \right) }^3}{n\,
      \left( 1 + 2\,n + 2\,n^2 \right) },&
b_{2n+2}&=\frac{1 + n}{1 + 2\,n + 2\,n^2}.&
\end{align*}
}
Then the odd and even parts of $G$ tend to different values and
$G$ converges in the general sense.
\end{example}
}
\begin{proof}
 It is easy to check that the numerators $A_{n}$ and denominators
$B_{n}$ satisfy
 \begin{align*}
A_{2n-1}&=n+1, &A_{2n} &= n + 3 n^{2},\\
B_{2n-1}&=n,   &B_{2n} &= n^{2}.
\end{align*}
Thus the odd approximants tend to 1 and the even approximants tend
to 3.  Observe
that
\begin{align*}
\frac{A_{2n}+w_{2n}A_{2n-1}}{B_{2n}+w_{2n}B_{2n-1}} &=
\frac{n + 3\,n^2 + \left( 1 + n \right) \,w_{2n}}{n^2 + n\,w_{2n}},\\
&\phantom{as}\\
\frac{A_{2n+1}+w_{2n+1}A_{2n}}{B_{2n+1}+w_{2n+1}B_{2n}} &=
\frac{2+n +  \left( n +3\, n^{2} \right) \,w_{2n+1}}{1+n +
n^{2}\,w_{2n+1}}.
\end{align*}
Each of these expressions converges to 3, when, for example,
 $\{w_{n}\}$ is the
constant sequence with value 1 and when it is the constant
sequence with value 2. Thus the continued fraction converges
generally to 3.
\end{proof}

It is therefore desirable to have criteria, based on the partial
quotients of a continued fraction, for determining whether a
continued fraction  whose odd and even parts converge diverges in
the general sense.

 An example of a theorem on divergence in the general sense is the
Stern-Stolz Theorem, as extended by Lorentzen and Waadeland.
\begin{theorem}\label{STT}
$($\textbf{The Stern-Stolz Theorem} $($\cite{LW92}, p.94$))$ The
continued fraction  $b_{0}$ $+$ $K_{n=1}^{\infty}1/b_{n}$
diverges generally
 if $\sum |b_{n}| < \infty$. In fact,
\[
\lim_{n\to \infty}A_{2n+p}=P_{p}\not = \infty, \hspace{25pt}
\lim_{n\to \infty}B_{2n+p}=Q_{p}\not = \infty,
\]
for $p=0,1$, where
\[
P_{1}Q_{0}-P_{0}Q_{1}=1.
\]
\end{theorem}
However, if a continued fraction is not already of the form
$K_{n=1}^{\infty}1/b_{n}$, these $b_{n}$  may become quite
complicated once an equivalence transformation is applied to the
continued fraction to bring it to this form and it may not be so
easy to determine if the series $\sum |b_{n}|$ converges.

In the present paper, we prove a theorem which gives a simple
criterion, based on the partial quotients, for deciding if a continued fraction
diverges in the general sense, provided it is known that the
odd- and even parts converge and
whether these limits are equal.

We then apply our theorem to two classes of $q$-continued fraction
described in our paper \cite{BML03} to show that if $|q_{1}|>1$
and $H(q)$ is a continued fraction in either class, then either
$H(q_{1})$ converges or does not converge generally.

\section{A Theorem on Divergence in the General Sense }

We now prove the following theorem.
{\allowdisplaybreaks
\begin{theorem}\label{P:pr2}
Let the odd and even parts of the continued fraction $C = b_{0} +
K_{n=1}^{\infty}a_{n}/b_{n}$ converge to different limits. Further
suppose that there exist positive constants $c_{1}$, $c_{2}$ and
$c_{3}$ such that, for $i \geq 1$, {\allowdisplaybreaks
\begin{align}\label{con1}
c_{1} \leq |b_{i}| \leq c_{2}
\end{align}
} and {\allowdisplaybreaks
\begin{align}\label{con2a}
\left|\frac{a_{2i+1}}{a_{2i}}\right| \leq c_{3}.
\end{align}
} Then $C$ does not converge generally.
\end{theorem}
} Remark: It might seem that Condition \ref{con1} prevents the
application of this theorem to  continued fractions $K_{n=1}^{\infty}a_{n}/b_{n}$ in which the
$b_{n}$ become unbounded but a similarity transformation to put
the continued fraction in the form $b_{0}+K_{n=1}^{\infty}c_{n}/1$
removes this difficulty.

\emph{Proof of Theorem \ref{P:pr2}}.
Let the $i$-th approximant of
$C=b_{0} + K_{n=1}^{\infty}a_{n}/b_{n}$  be
denoted by $A_{i}/B_{i}$.
Suppose the odd approximants tend to $f_1$ and that
the even approximants tend to $f_{2}$. Further suppose that
$C$ converges generally to $f \in \hat{\mathbb{C}}$ and that
$\{v_{n}\}$, $\{w_{n}\} \subset \hat{\mathbb{C}}$  are
two sequences such that
\begin{align*}
\lim_{n \to \infty}\frac{A_{n}+v_{n}A_{n-1}}
     {B_{n}+v_{n}B_{n-1}}=
\lim_{n \to \infty}\frac{A_{n}+w_{n}A_{n-1}}
     {B_{n}+w_{n}B_{n-1}}= f
\end{align*}
and
\begin{align*}
\liminf_{n \to \infty} d(v_{n},w_{n}) > 0.
\end{align*}
It will be shown that these two conditions lead to a contradiction.
Suppose first that $|f|< \infty$ and, without loss of generality,
that $f \not = f_{1}$. (If $f = f_{1}$, then  $f \not = f_{2}$
and we proceed similarly).
We write
\begin{align*}
\frac{A_{n}+w_{n}A_{n-1}}
     {B_{n}+w_{n}B_{n-1}}= f+ \gamma_{n},\phantom{asdsad}
\frac{A_{n}+v_{n}A_{n-1}}
     {B_{n}+v_{n}B_{n-1}}= f + \gamma_{n}^{'},
\end{align*}
where $ \gamma_{n} \to 0$ and $ \gamma_{n}^{'} \to 0$ as
$n \to \infty$.
 By assumption it follows that
$A_{2n} = B_{2n}(f_2 + \alpha_{2n})$ and
$A_{2n+1} = B_{2n+1}(f_1 + \alpha_{2n+1})$, where
$\alpha_{i} \to 0$ as $i \to \infty$. Then
\begin{align*}
\frac{A_{2n}+w_{2n}A_{2n-1}}
     {B_{2n}+w_{2n}B_{2n-1}} &=
\frac{B_{2n}(f_2 + \alpha_{2n}) +w_{2n}B_{2n-1}(f_1 + \alpha_{2n-1})}
     {B_{2n}+w_{2n}B_{2n-1}} \\
&= f + \gamma_{2n}.
\end{align*}
By simple algebra we have
\begin{align*}
w_{2n}=\frac{{B_{2\,n}}\,\left( -f + {f_2} + {{\alpha }_{2\,n}} -
      {{\gamma }_{2\,n}} \right) }{{B_{ 2\,n-1}}\,
    \left( f - {f_1} - {{\alpha }_{ 2\,n-1}} +
      {{\gamma }_{2\,n}} \right) }.
\end{align*}
Similarly,
\begin{align*}
v_{2n}=\frac{{B_{2\,n}}\,\left( -f + {f_2} + {{\alpha }_{2\,n}} -
      {{\gamma }_{2\,n}^{'}} \right) }{{B_{ 2\,n-1}}\,
    \left( f - {f_1} - {{\alpha }_{ 2\,n-1}} +
      {{\gamma }_{2\,n}^{'}} \right) }.
\end{align*}
Note that $B_{2\,n}$, $B_{2\,n-1} \not = 0$ for $n$ sufficiently
large, since the odd and even parts of the continued fraction
converge. If $f \not = f_{2}$, then
\begin{align*}
\lim_{n \to \infty}d(v_{2n},w_{2n}) \leq
\lim_{n \to \infty}\frac{|v_{2n} - w_{2n}|}{|w_{2n}|} =0.
\end{align*}
Hence $f = f_{2}$,
\begin{align*}
w_{2n}=\frac{{B_{2\,n}}\,\left( {{\alpha }_{2\,n}} -
      {{\gamma }_{2\,n}} \right) }{{B_{ 2\,n-1}}\,
    \left( f - {f_1} - {{\alpha }_{ 2\,n-1}} +
      {{\gamma }_{2\,n}} \right) }
\end{align*}
and
\begin{align*}
v_{2n}=\frac{{B_{2\,n}}\,( {{\alpha }_{2\,n}} -
      {{\gamma }_{2\,n}^{'}} ) }{{B_{ 2\,n-1}}\,
    \left( f - {f_1} - {{\alpha }_{ 2\,n-1}} +
      {{\gamma }_{2\,n}^{'}} \right) }.
\end{align*}
 Now we show  that
\[
\lim_{n \to \infty} \left|\frac{B_{2n}}{B_{2n-1}}\right| = \infty.
\]
For if not, then there
is a sequence $\{n_{i}\}$ and a positive constant $M$ such that
 $|B_{2n_{i}}/B_{2n_{i}-1}| \leq M$ for all $n_{i}$, and then
\begin{align*}
&\lim_{i \to \infty} d(v_{2n_{i}},w_{2n_{i}})
\leq \lim_{i \to \infty} \left| v_{2n_{i}}-w_{2n_{i}}
\right|\\
&\leq\lim_{i \to \infty} M\left|\frac{ {{\alpha }_{2\,n_{i}}} -
      {{\gamma }_{2\,n_{i}}^{'}}  }{
     f - {f_1} - {{\alpha }_{ 2\,n_{i}-1}} +
      {{\gamma }_{2\,n_{i}}^{'}}  }-
\frac{ {{\alpha }_{2\,n_{i}}} -
      {{\gamma }_{2\,n_{i}}}  }{
     f - {f_1} - {{\alpha }_{ 2\,n_{i}-1}} +
      {{\gamma }_{2\,n_{i}}}  }
\right|=0.
&\phantom{as}
\end{align*}
Similarly, after substituting $f_{2}$ for $f$, we have that
\begin{align*}
w_{2n+1} = \frac{B_{2n+1}}{B_{2n}}
\left( \frac{f_{1}-f_{2} + \alpha_{2n+1} -\gamma_{2n+1}}
{\gamma_{2n+1}-\alpha_{2n}}\right)
\end{align*}
and
\begin{align*}
v_{2n+1} = \frac{B_{2n+1}}{B_{2n}}
\left( \frac{f_{1}-f_{2} + \alpha_{2n+1} -\gamma_{2n+1}^{'}}
{\gamma_{2n+1}^{'}-\alpha_{2n}}\right).
\end{align*}
We now show that
\[
\lim_{n \to \infty} \left|\frac{B_{2n+1}}{B_{2n}}\right| = 0.
\]
If not, then there
is a sequence $\{n_{i}\}$ and some $M>0$ such that
 $|B_{2n_{i}+1}/B_{2n_{i}}|$ $\geq M$ for all $n_{i}$. Then
$\lim_{i \to \infty} w_{2n_{i}+1}=\lim_{i \to \infty} v_{2n_{i}+1}
= \infty$ and $\lim_{i \to \infty}$ $ d(v_{2n_{i}+1},w_{2n_{i}+1})
=0$.

Finally, we show that it is impossible to have both
$\lim_{n \to \infty}|B_{2n+1}/B_{2n}| = 0$ and
$\lim_{n \to \infty}$ $|B_{2n}/B_{2n-1}| = \infty$. For ease of
notation let
$B_{n}/B_{n-1}$ be denoted by $r_{n}$, so that
$r_{2n} \to \infty$ and $r_{2n+1} \to 0$, as $n \to \infty$.
From the recurrence relations for the $B_{i}$'s, namely,
$B_{i}=b_{i}B_{i-1}+a_{i}B_{i-2}$,
  we have
\begin{align*}
r_{2n}(r_{2n+1}-b_{2n+1}) &= a_{2n+1}
\end{align*}
and
{\allowdisplaybreaks
\begin{align*}
r_{2n-1}(r_{2n}-b_{2n}) &= a_{2n}.
\end{align*}
}
Thus
{\allowdisplaybreaks
\begin{align*}
\frac{r_{2n}}{r_{2n}-b_{2n}}
&=\frac{a_{2n+1}r_{2n-1}}{a_{2n}(r_{2n+1}-b_{2n+1})},
\end{align*}
}
and by \eqref{con1} and \eqref{con2a}
the left side tends to 1 and the right side tends to 0,
as $n \to \infty$, giving the
required contradiction.

If $f = \infty$, then we write
\begin{align*}
\frac{A_{n}+w_{n}A_{n-1}}
     {B_{n}+w_{n}B_{n-1}}&=\frac{1}{\gamma_{n}},\\
&\phantom{as}\\
\frac{A_{n}+v_{n}A_{n-1}}
     {B_{n}+v_{n}B_{n-1}}&=\frac{1}{\gamma_{n}'},
\end{align*}
where
$\lim_{n \to \infty}\gamma_{n} = \lim_{n \to \infty}\gamma_{n}' =0$.
With the $\alpha_{i}$'s
as above we find that
\begin{align*}
w_{2n} = - \frac{{B_{2\,n}}\,
      \left( -1 + {f_2}\,{{\gamma }_{2\,n}} +
        {{\alpha }_{2\,n}}\,{{\gamma }_{2\,n}} \right) }
      {{B_{ 2\,n-1}}\,\left( -1 +
        {f_1}\,{{\gamma }_{2\,n}} +
        {{\alpha }_{ 2\,n+1}}\,{{\gamma }_{2\,n}} \right) }
\end{align*}
and
\begin{align*}
v_{2n} = - \frac{{B_{2\,n}}\,
      \left( -1 + {f_2}\,{{\gamma }_{2\,n}^{'}} +
        {{\alpha }_{2\,n}}\,{{\gamma }_{2\,n}^{'}} \right) }
      {{B_{ 2\,n-1}}\,\left( -1 +
        {f_1}\,{{\gamma }_{2\,n}^{'}} +
        {{\alpha }_{ 2\,n+1}}\,{{\gamma }_{2\,n}^{'}} \right) }.
\end{align*}
In this case it follows easily that
$\lim_{n \to \infty}d(w_{2n}, v_{2n}) = 0$.
\begin{flushright}
$\Box$
\end{flushright}

\section{Application to $q$-Continued Fractions}
In \cite{BML03}, one type of continued fraction we considered was of the form
{\allowdisplaybreaks
\begin{align*}
G(q):&=1 +
K_{n=1}^{\infty}\frac{a_{n}(q)}{1}:=1+
\frac{f_{1}(q^{0})}{1}
\+\cds \+
\frac{f_{k}(q^{0})}{1}\\
&\+
 \frac{f_{1}(q^{1})}{1}
\+\cds \+
\frac{f_{k}(q^{1})}{1}
\+\cds \+
\frac{f_{1}(q^{n})}{1}
\+ \cds\+
\frac{f_{k}(q^{n})}{1} \+\cds , \notag
\end{align*}
}
where $f_{s}(x)\in \mathbb{Z}[q][x]$, for $1 \leq s
\leq k$. Thus, for $n\geq 0$ and $1 \leq s \leq k$,
{\allowdisplaybreaks
\begin{align}\label{con4}
&a_{nk+s}(q)=f_{s}(q^{n}).&
\end{align}
}
Many well-known $q$-continued fractions, including the
Rogers-Ramanujan continued fraction at  \eqref{rrcf} and the three Ramanujan-Selberg
continued fractions studied by Zhang in \cite{Z91}, namely,
{\allowdisplaybreaks
\begin{align*}
S_{1}(q):= 1 + \frac{q}{1}
\+
 \frac{q+q^{2}}{1}
\+
 \frac{q^{3}}{1}
\+
 \frac{q^{2}+q^{4}}{1}
\+
\cds ,
\end{align*}
}
\begin{align*}
S_{2}(q):=
1 +
 \frac{q+q^{2}}{1}
\+
 \frac{q^{4}}{1}
\+
 \frac{q^{3}+q^{6}}{1}
\+
 \frac{q^{8}}{1}
\+
\cds ,
\end{align*}
and
\begin{align*}
S_{3}(q):=
1 +
 \frac{q+q^{2}}{1}
\+
 \frac{q^{2}+q^{4}}{1}
\+
 \frac{q^{3}+q^{6}}{1}
\+
 \frac{q^{4}+q^{8}}{1}
\+
\cds ,
\end{align*}
\begin{flushleft}
are of this form, with $k$ at most 2.
Following the example of these four continued fractions,
we made the additional assumptions that,
 for $i \geq 1$,
\end{flushleft}
\begin{align}\label{con2}
 \text{degree}(a_{i+1}(q))= \text{degree}(a_{i}(q))+ C_{3}, \\
&\phantom{as} \notag
\end{align}
where $C_{3}$ is a fixed positive integer, and
that
 all of the polynomials $a_{n}(q)$ had the same
leading coefficient.  The odd- and even parts of each of the four
continued fractions above converge for $|q|>1$, (see \cite{ABJL92}
and \cite{Z91}, where the authors also determined the limits).  In
\cite{BML03}, we extended these results on convergence outside the
unit circle to the  class of continued fractions described above.
We proved the following theorem \cite{BML03}:
\begin{theorem}\label{T4}\cite{BML03}
 Suppose $G(q) = 1 + K_{n=1}^{\infty}a_{n}(q)/1$
is such that the $a_{n}:= a_{n}(q)$ satisfy \eqref{con4} and
 \eqref{con2}. Suppose further that each $a_{n}(q)$ has the same
 leading coefficient. If  $|q|>1$ then the odd and
even parts of $G(q)$ both converge.
\end{theorem}
It is now an easy matter to apply our Theorem \ref{P:pr2}
to the continued fractions of Theorem \ref{T4}
to conclude that for each $q$ outside the unit circle, either the continued fraction
converges or does not converge generally.
As an illustration we have the following example.
\begin{example}\label{exa}
Let
\begin{multline*}
G_{1}(q)= 1 + \frac{6q}{1} \+
 \frac{3q^{2}+7q}{1}
\+
 \frac{3q^{3}+5q^{2}}{1}
\+
 \frac{q^{4}+7q^{3}+3 q+2}{1}
 \+\\
\frac{q^{5}+3 q^{4}+2q^{3}}{1} \+
 \frac{q^{6}+2q^{5}+7 q^{3}}{1}
\+
 \frac{q^{7}+7q^{5}}{1}
\+
 \frac{q^{8}+7q^{6}+3q^{3}+2q  }{1}
\+\,\cds \\
\cds \+ \frac{q^{4n+1}+3 q^{3n+1}+2q^{2n+1}}{1} \+
 \frac{q^{4n+2}+2q^{3n+2}+7 q^{2n+1}}{1} \\
\+
 \frac{q^{4n+3}+5q^{3n+2}+2q^{2n+3}}{1}
\+
 \frac{q^{4n+4}+7q^{3n+3}+3q^{2n+1}+2q^{n}  }{1}
\+\,\cds.
\end{multline*}
 If $|q|>1$, then the odd and even parts of $G_{1}(q)$
converge. If the odd and even parts are not equal, then $G_{1}(q)$ does not converge generally.
\end{example}

In \cite{BML03} we also studied continued fractions of the form
\begin{align*}
G(q):=b_{0}(q) +&
K_{n=1}^{\infty}\frac{a_{n}(q)}{b_{n}(q)}\\
:=g_{0}(q^{0})+
&\frac{f_{1}(q^{0})}{g_{1}(q^{0})}
 \+\,\cds \+
\frac{f_{k-1}(q^{0})}{g_{k-1}(q^{0})}
\+
\frac{f_{k}(q^{0})}{g_{0}(q^{1})} \notag \\
\+
&\frac{f_{1}(q^{1})}{g_{1}(q^{1})}
 \+\,\cds \+
\frac{f_{k-1}(q^{1})}{g_{k-1}(q^{1})}
\+
\frac{f_{k}(q^{1})}{g_{0}(q^{2})} \+\notag \\
  \cds \+
\frac{f_{k}(q^{n-1})}{g_{0}(q^{n})}
\+
&\frac{f_{1}(q^{n})}{g_{1}(q^{n})}
 \+\,\cds \+
\frac{f_{k-1}(q^{n})}{g_{k-1}(q^{n})} \+
\frac{f_{k}(q^{n})}{g_{0}(q^{n+1})} \+\cds \notag
\end{align*}
where $f_{s}(x), g_{s-1}(x) \in \mathbb{Z}[q][x]$, for $1 \leq s
\leq k$. Thus, for $n\geq 0$ and $1 \leq s \leq k$,
\begin{align}\label{con4ab}
&a_{nk+s}(q)=f_{s}(q^{n}),& &b_{nk+s-1}(q)=g_{s-1}(q^{n}).&
\end{align}
An example of a continued fraction of this type is the G\"{o}llnitz-Gordon continued
fraction ($k=1$),
\begin{align*}
GG(q):=
1+q +
 \frac{q^{2}}{1+q^{3}}
\+
 \frac{q^{4}}{1+q^{5}}
\+
 \frac{q^{6}}{1+q^{7}}
\+\,\cds.
\end{align*}
We restricted the type of continued fraction examined as follows.
We supposed that degree $(a_{1}(q))=r_{1}$,
degree $(b_{0}(q))= r_{2}$, and that,
 for $i \geq 1$,
\begin{align}\label{con5}
 \text{degree}(a_{i+1}(q))&= \text{degree}(a_{i}(q))+ a, \\
\text{degree}(b_{i}(q))&= \text{degree}(b_{i-1}(q))+ b, \notag
\end{align}
where $a$ and $b$ are fixed positive integers and $r_{1}$ and
$r_{2}$ are non-negative integers. Condition \ref{con5} means
that, for $n \geq 1$,
 $a_{n}(q)$ has degree equal to $(n-1)a + r_{1}$ and that
$b_{n}(q)$ has degree equal to $n\,b + r_{2}$. We also supposed that each $a_{n}(q)$
has the same leading coefficient $L_{a}$ and that each $b_{n}(q)$
has the same leading coefficient $L_{b}$.

For such continued fractions we had the following theorem
\cite{BML03}:
\begin{theorem}\label{T:p2}\cite{BML03}
Suppose $G(q) = b_{o} + K_{n=1}^{\infty}a_{n}(q)/b_{n}(q)$
is such that the $a_{n}:= a_{n}(q)$ and the
 $b_{n}:= b_{n}(q)$ satisfy \eqref{con4ab} and \eqref{con5}.
 Suppose further that
 each $a_{n}(q)$ has the same leading
coefficient $L_{a}$ and that each $b_{n}(q)$ has the same leading
coefficient  $L_{b}$. If $2b>a$ then $G(q)$ converges everywhere
outside the unit circle. If $2b =a$,
 then  $G(q)$  converges outside the unit circle
 to values in $\hat{\mathbb{C}}$,
except possibly at points $q$ satisfying
$q^{b-r_{1}+2r_{2}} \in
\left[-4\,L_{a}/L_{b}^2, 0\right) $ or
$\left(0,-4\,L_{a}/L_{b}^2\right]$, depending on the sign of $L_{a}$.
If $2b<a$, then the odd and even parts of
$G(q)$ converge everywhere outside the unit circle.
\end{theorem}
Remark: Both our Theorem \ref{T4} and \ref{T:p2}
were  derived from theorems  on limit--periodic continued fractions and
give stronger results than can be derived from applying simple convergence criteria such
as  Worpitzky's Theorem.

Once again it is  easy to apply our Theorem \ref{P:pr2}
to the continued fractions of Theorem \ref{T:p2}
to conclude, in the case $2b<a$, that for each $q$ outside the unit circle, either the continued fraction
converges or does not converge generally. As an illustration we have the following example.
\begin{example}\label{exab}
Let
 {\allowdisplaybreaks
\begin{multline*}
G_{2}(q):= q+2 \,+\\
\frac{q^{3}+5q^{2}}{q^{2}+2}
\+
 \frac{q^{6}+2q^{4}+7q^{2}}{q^{3}+2}
\+
 \frac{q^{9}+2q^{6}+5q^{4}}{q^{4}+2}
\+
 \frac{q^{12}+7q^{6}+3 q^{2}+2}{q^{5}+q+1}
 \+\\
\frac{q^{15}+3 q^{8}+2q^{6}}{q^{6}+q^{2}+1}
 \+
 \frac{q^{18}+2q^{10}+7 q^{6}}{q^{7}+q^{2}+1}
\+
 \frac{q^{21}+7q^{10}}{q^{8}+q^{3}}
\+
 \frac{q^{24}+7q^{12}+3q^{6}+2q^{2}  }{q^{9}+q^{2}+1}\\
\+ \cds \+
 \frac{q^{12n+3}+3 q^{6n+2}+2q^{4n+2}}{q^{4n+2}+q^{2n}+1}
\+
 \frac{q^{12n+6}+2q^{6n+4}+7 q^{4n+2}}{q^{4n+3}+q^{2n}+1}\\
\+
 \frac{q^{12n+9}+5q^{6n+4}+2q^{4n+6}}{q^{4n+4}+q^{3n}+1}
\+
 \frac{q^{12n+12}+7q^{6n+6}+3q^{4n+2}+2q^{2n}  }{q^{4(n+1)+1}+q^{n+1}+1}
\+\,\cds .
\end{multline*}
}  If $|q|>1$, then the odd and even parts of $G_{2}(q)$ converge.
If the odd and even parts are not equal, then $G_{2}(q)$ does not converge generally.
\end{example}

\section{Continued fractions whose odd and even parts tend to different limits}
Since our Theorem \ref{P:pr2} deals with continued fractions whose odd and even
parts converge to different values, it is desirable to know something about the
form of such continued fractions. We have the following theorem.
\begin{theorem}\label{todev}
Suppose the odd and even parts of the continued fraction $K_{n=1}^{\infty}a_{n}/1$
converge to different values. Then $\lim_{n \to \infty}|a_{n}| = \infty$.
\end{theorem}
We need two preliminary results.
\begin{lemma}\label{lem1}
Suppose $\{K_{n}\}_{n=1}^{\infty}$ is the sequence of  classical
approximants of the continued fraction $K_{n=1}^{\infty}a_{n}/1$, where
$a_{n} \not = 0$, for $n \geq 1$.
If the continued fraction $K_{n=1}^{\infty}c_{n}/1$  also has
$\{K_{n}\}_{n=1}^{\infty}$  as it  sequence of  classical
approximants and $c_{n} \not = 0$, for $n \geq 1$, then $a_{n}=c_{n}$ for $n\geq 1$.
\end{lemma}
\begin{proof}
Elementary.
\end{proof}
We also use the following result,
proved by Daniel Bernoulli in 1775 \cite{B75} (see, for example, \cite{K63}, pp. 11--12).
{\allowdisplaybreaks
\begin{proposition}\label{pber}
Let $\{K_{0},K_{1}, K_{2},\ldots\}$ be a sequence of complex numbers such that
$K_{i}\not = K_{i-1}$,
for $i=1,2,\ldots$.
Then  $\{K_{0},K_{1}, K_{2},\ldots\}$  is the sequence of approximants of the
continued fraction
\begin{align*}
K_{0}+\frac{K_{1}-K_{0}}{1}
\+
\frac{K_{1}-K_{2}}{K_{2}-K_{0}}
\+
\frac{(K_{1}-K_{0})(K_{2}-K_{3})}
            {K_{3}-K_{1}}
\+\\
\cds
\+
\frac{(K_{n-2}-K_{n-3})(K_{n-1}-K_{n})}
            {K_{n}-K_{n-2}}
\+
\cds \notag \\
\thicksim
K_{0}+\frac{K_{1}-K_{0}}{1}
\+
\frac{\frac{K_{1}-K_{2}}{K_{2}-K_{0}}}{1}
\+
\frac{\frac{(K_{1}-K_{0})(K_{2}-K_{3})}
            {(K_{2}-K_{0})(K_{3}-K_{1})}}{1}
\+  \notag \\
\cds
\+
\frac{\frac{(K_{n-2}-K_{n-3})(K_{n-1}-K_{n})}
            {(K_{n-1}-K_{n-3})(K_{n}-K_{n-2})}}{1}
\+
\cds .  \notag
\end{align*}
\end{proposition}
} \emph{Proof of Theorem \ref{todev}.} Let
$\{K_{n}\}_{n=1}^{\infty}$ denote the sequence of classical
approximants of the continued fraction $K_{n=1}^{\infty}a_{n}/1$.
By assumption there exist $\alpha \not = \beta \in \mathbb{C}$
such that
\begin{align*}
&\lim_{n \to \infty} K_{2n} = \alpha ,& &\lim_{n \to \infty} K_{2n+1} = \beta .&
\end{align*}
Hence there exist two null sequences $\{\alpha_{n}\}_{n=0}^{\infty}$ and
$\{\beta_{n}\}_{n=0}^{\infty}$ such that
\begin{align}\label{keqs}
& K_{2n} = \alpha + \alpha_{n},& &K_{2n+1} = \beta + \beta_{n}.&
\end{align}
By Lemma \ref{lem1}, Proposition \ref{pber} and \eqref{keqs}, it follows that
\begin{align*}
a_{2n}&=\frac{(K_{2n-2}-K_{2n-3})(K_{2n-1}-K_{2n})}
            {(K_{2n-1}-K_{2n-3})(K_{2n}-K_{2n-2})}\\
&=\frac{(\alpha + \alpha_{n-1}-\beta - \beta_{n-2})(\beta + \beta_{n-1}-\alpha - \alpha_{n})}
{(\beta_{n-1}-\beta_{n-2})(\alpha_{n}-\alpha_{n-1})}.
\end{align*}
Since $\alpha \not = \beta$ and $\{\alpha_{n}\}_{n=0}^{\infty}$ and
$\{\beta_{n}\}_{n=0}^{\infty}$ are null sequences, it follows that
$\lim_{n \to \infty}|a_{2n}| = \infty$. That $\lim_{n \to \infty}|a_{2n-1}| = \infty$
follows similarly.
\begin{flushright}
$\Box$
\end{flushright}

\section{Concluding Remarks}

Let $m \geq 2$ be a positive integer.
A  continued fraction for which the odd- and even parts tend to different
limits may be regarded as a special case  ($m=2$) of continued fractions
for which the sequence of approximants in each arithmetic progression
modulo $m$ tends to a different limit. We will investigate such continued fractions in a later paper
and  also look at the question of whether or not they
converge in the general sense.

We close with a question. Does there exist a continued fraction
$K_{n=1}^{\infty}a_{n}/1$ whose odd and even parts converge to different values,
 for which the sequence $\{a_{2n+1}/a_{2n}\}$ is bounded and
whose Stern--Stolz series diverges?
This would mean that our Theorem  \ref{P:pr2} could show divergence in the general
sense for a continued fraction that the Stern-Stolz Theorem could not be applied to.

On the other hand, it may be that if $\{a_{n}\}$ is any sequence of non-zero
complex numbers  such that
 the sequence $\{a_{2n+1}/a_{2n}\}$ is bounded
and the continued fraction
$K_{n=1}^{\infty}a_{n}/1$ is such that its odd and even parts converge to different values,
then the Stern--Stolz series for   $K_{n=1}^{\infty}a_{n}/1$ converges.
A proof of this would be interesting. In this latter  situation
our Theorem \ref{P:pr2} does not give anything new and may just be easier to apply
to certain types of continued fraction.

{\allowdisplaybreaks

}


\begin{thebibliography}{99}

\bibitem{ABJL92}
Andrews, G. E.; Berndt, Bruce C.; Jacobsen, Lisa; Lamphere, Robert L.
\emph{The continued fractions
found in the unorganized portions of Ramanujan's notebooks}.
Mem. Amer. Math. Soc. \textbf{99} (1992), no. 477, vi+71pp

\bibitem{B75}   Daniel Bernoulli,
  \emph{  Disquisitiones ulteriores de indola fractionum
continuarum, Novi comm.}, Acad. Sci. Imper. Petropol. \textbf{20} (1775)



\bibitem{BML01}
Bowman, D; Mc Laughlin, J
\emph{On the Divergence of the Rogers-Ramanujan Continued Fraction on
the Unit Circle. }
To appear in the Transactions of the American Mathematical Society.


\bibitem{BML02a}
Bowman, D; Mc Laughlin, J
\emph{The   Convergence Behavior of  $q$-Continued Fractions on the Unit
Circle }
To appear in The Ramanujan Journal.




\bibitem{BML02}
Bowman, D; Mc Laughlin, J
\emph{On the divergence in the general sense of $q$-continued fraction on the unit circle.
 }
Commun. Anal. Theory Contin. Fract. 11 (2003), 25--49.



\bibitem{BML03}
Bowman, D; Mc Laughlin, J
\emph{The Convergence and  Divergence of  $q$-Continued Fractions
outside
the Unit Circle }
To appear in The Rocky Mountain Journal of Mathematics.






\bibitem{J86}
Jacobsen, Lisa
\emph{General convergence of continued fractions}.
Trans. Amer. Math. Soc. \textbf{294} (1986), no. 2, 477--485.


\bibitem{K63}   Alexey Nikolaevitch Khovanskii,
\emph{ The application of continued fractions and their
generalizations to problems in approximation theory}, Translated by Peter Wynn,
P. Noordhoff N. V., Groningen 1963 xii + 212 pp.


\bibitem{LW92}
Lorentzen, Lisa; Waadeland, Haakon
\emph{Continued fractions with applications}. Studies in
Computational Mathematics, 3. North-Holland Publishing Co.,
Amsterdam, 1992, pp 35--36.


\bibitem{Z91}
Zhang, Liang Cheng(1-IL)
$q$-difference equations and Ramanujan-Selberg continued fractions.
Acta Arith. 57 (1991), no. 4, 307--355.


\end{thebibliography}
\end{document}